\documentclass{amsart}

\usepackage{amsmath,amssymb,amsthm}
\usepackage{url}
\usepackage{enumerate}

\newtheorem{theorem}{Theorem}[section]
\newtheorem{lemma}[theorem]{Lemma}

\newtheorem{proposition}[theorem]{Proposition}
\newtheorem{fact}[theorem]{Fact}

\theoremstyle{definition}
\newtheorem{definition}[theorem]{Definition}
\newtheorem{example}[theorem]{Example}
\newtheorem{remark}[theorem]{Remark}
\newtheorem{question}[theorem]{Question}

\newcommand{\C}{\mathbb C}
\newcommand{\M}{\mathbb M}
\newcommand{\N}{\mathbb N}
\newcommand{\Q}{\mathbb Q}
\newcommand{\R}{\mathbb R}
\newcommand{\Z}{\mathbb Z}

\def\tp{\operatorname{tp}}
\def\wt{\operatorname{wt}}
\def\rk{\operatorname{rk}}
\def\dpr{\operatorname{dp}}
\def\eq{\operatorname{eq}}
\def\acl{\operatorname{acl}}
\def\Th{\operatorname{Th}}
\def\Aut{\operatorname{Aut}}

\def\Ind{\setbox0=\hbox{$x$}\kern\wd0\hbox to 0pt{\hss$\mid$\hss}
\lower.9\ht0\hbox to 0pt{\hss$\smile$\hss}\kern\wd0}
\def\Notind{\setbox0=\hbox{$x$}\kern\wd0\hbox to 0pt{\mathchardef
\nn=12854\hss$\nn$\kern1.4\wd0\hss}\hbox to
0pt{\hss$\mid$\hss}\lower.9\ht0 \hbox to 0pt{\hss$\smile$\hss}\kern\wd0}
\def\ind{\mathop{\mathpalette\Ind{}}}
\def\nind{\mathop{\mathpalette\Notind{}}}

\title{Stable groups and expansions of $(\Z,+,0)$}

\date{July 25, 2017}

\author{Gabriel Conant}

\author{Anand Pillay}
\thanks{The second author was supported by NSF grant DMS-136072.}

\address{Department of Mathematics\\
University of Notre Dame\\
Notre Dame, IN, 46656, USA}


\begin{document}

\begin{abstract}
We show that if $G$ is a sufficiently saturated stable group of finite weight with no infinite, infinite-index, chains of definable subgroups, then $G$ is superstable of finite $U$-rank. Combined with recent work of Palac\'{i}n and Sklinos, we conclude that $(\Z,+,0)$ has no proper stable expansions of finite weight. A corollary of this result is that if $P\subseteq\Z^n$ is definable in a finite dp-rank expansion of $(\Z,+,0)$, and $(\Z,+,0,P)$ is stable, then $P$ is definable in $(\Z,+,0)$. In particular, this answers a question of Marker on stable expansions of the group of integers by sets definable in Presburger arithmetic. 
\end{abstract}

\maketitle

\section{Introduction and Summary of Main Results}
\setcounter{theorem}{0}
\numberwithin{theorem}{section}

The work in this paper is motivated by questions surrounding first-order expansions of the group $(\Z,+,0)$, which are well-behaved with respect to some notion of model theoretic tameness (e.g. stability or NIP). The group $(\Z,+,0)$ is a well-known example of a stable group, and so this program is a natural analog of the very fruitful study of ``tame" (e.g. o-minimal or NIP) expansions of the real ordered field $(\R,+,\cdot,<,0)$. Expansions of $(\Z,+,0)$ have emerged in the context of definable subgroups of finitely generated free groups, as well as the general growing industry of research on ordered abelian groups satisfying notions of tameness coming from dp-rank in NIP first-order theories (e.g. \cite{DoGo}, \cite{FlGu}, \cite{SiDP}). We will provide more detail on these contexts toward the end of the introduction. For now, we state an explicit question, originally asked by Marker in 2011.

\begin{question}[Marker]\label{ques:Mark}
Is there a set $P\subseteq\Z^n$, definable in Presburger arithmetic $(\Z,+,<,0)$, such that $(\Z,+,0,P)$ is a proper stable expansion of $(\Z,+,0)$?
\end{question}

The focus on Presburger arithmetic in the previous question is not unnatural. Indeed, $(\Z,+,<,0)$ is an ordered structure, and thus unstable, but is still well understood and very well behaved model theoretically (to be precise, its theory is NIP of dp-rank $1$ \cite{DGL}). Our first main result will show that, in fact, these model theoretic notions completely control the answer to Marker's question.

\begin{theorem}\label{thm:mainZ}
If $P\subseteq\Z^n$ is definable in a finite dp-rank expansion of $(\Z,+,0)$, and $(\Z,+,0,P)$ is stable, then $P$ is definable in $(\Z,+,0)$. 
\end{theorem}

The notion of dp-rank in NIP theories has been an important tool in extending the work of stability theory to the unstable setting (see, e.g., \cite{Sibook}), and so Theorem \ref{thm:mainZ} establishes a fundamental fact about the behavior of NIP expansions of $(\Z,+,0)$. The proof of this theorem will be obtained from a more general result on stable groups (Theorem \ref{thm:main} below), combined with the following result of Palac\'{i}n and Sklinos \cite{PaSk}.

\begin{fact}\cite{PaSk}\label{fact:PaSk}
$(\Z,+,0)$ has no proper stable expansions of finite $U$-rank.
\end{fact}

We emphasize that, \emph{a priori}, Fact \ref{fact:PaSk} alone is not sufficient to answer Marker's question, or obtain Theorem \ref{thm:mainZ}. In particular, while the dp-rank of a complete theory is \emph{bounded above} by its $U$-rank, there is no further general relationship between these two ranks. Indeed, there are stable groups of dp-rank $1$ and infinite or undefined $U$-rank (see Example \ref{ex:main}). Therefore, the work involved in proving Theorem \ref{thm:mainZ} consists of showing that if a stable expansion of $(\Z,+,0)$ has finite dp-rank, then it must have finite $U$-rank. In fact, we will obtain this conclusion from a general characterization of superstable groups of finite $U$-rank, which exploits the notion of \emph{weight} in stable theories. Before stating this result, we clarify the following terminology (full definitions are given in Section \ref{sec:prelim}).

Let $G$ be a group definable in a complete theory $T$. Unless otherwise stated, we assume $G$ is evaluated in a sufficiently saturated monster model. The \emph{$U$-rank of $G$}, denoted by $U(G)$, is the supremum of the $U$-ranks of types containing a formula defining $G$. Replacing $U$-rank with \emph{weight}, we similarly define the \emph{weight of $G$}, denoted by $\wt(G)$. We say $G$ is \emph{stable} if $T$ is stable. We let $<_\infty$ denote the partial order on groups given by: $H<_\infty K$ if $H\leq K$ and $[K:H]=\infty$. The \emph{length} of a finite chain $K_0<_\infty\ldots<_\infty K_n$ is $n$. If $G$ is superstable of finite $U$-rank then, by well-known facts, $G$ necessarily has finite weight and no infinite $<_\infty$-chains of definable subgroups (see \cite[Theorem 19.9]{Pobook} and \cite[Corollary III.8.2]{BeLa}). Our second main result is that these conditions are also sufficient.

\begin{theorem}\label{thm:main}
If $G$ is stable then the following are equivalent.
\begin{enumerate}[$(i)$]
\item $G$ is superstable of finite $U$-rank.
\item $G$ has finite weight and no infinite $<_\infty$-chains of definable subgroups.
\item $G$ has finite weight and no infinite $<_\infty$-chains of definable normal subgroups.
\end{enumerate}
\end{theorem}

Theorem \ref{thm:main} will be obtained as an immediate consequence of the following more detailed statement, which also gives an upper bound on the $U$-rank of $G$.

\begin{theorem}\label{thm:all}
Let $G$ be a stable group of finite weight. If $G$ has no infinite $<_\infty$-chains of definable normal subgroups then:
\begin{enumerate}[$(i)$]
\item there is a uniform finite bound on the length of a $<_\infty$-chain of definable subgroups of $G$, and
\item if $n$ is the maximal length of a $<_\infty$-chain of definable normal subgroups of $G$, then $U(G)\leq n\wt(G)$.
\end{enumerate}
\end{theorem}

The proof of this theorem involves a new application of Zilber indecomposability in the setting of weight (see Lemma \ref{lem:main}). In Section \ref{sec:prelim}, we will also recall some classical examples showing that the upper bound in this result cannot be improved in general. All three theorems stated above are proved in Section \ref{sec:proofs}.

We end this section with a discussion of related work and open questions. The motivation for Question \ref{ques:Mark} partly arose from interest in the induced structure on proper definable subgroups of finitely-generated free groups, which are examples of stable groups \cite{Sela}. In particular, the maximal proper definable subgroups of such groups are exactly the centralizers of some  nontrivial element, and thus isomorphic \emph{as groups} to $(\Z,+,0)$ (see \cite{PiFFG}). Therefore, studying stable expansions of $(\Z,+,0)$ was seen as an alternate approach toward the unpublished result of Perin that the induced structure on centralizers in the free group is always a pure group. Another proof of this has been recently given by Byron and Sklinos \cite{BySk}.

Beyond this connection to the free group, there has been a recent flurry of interest in expansions of $(\Z,+,0)$. On the stable side, we have the following ambitious question (which is similar to a question of Goodrick quoted in \cite{PaSk}).

\begin{question}
Characterize the sets $P\subseteq\Z^n$ such that $(\Z,+,0,P)$ is stable.
\end{question}

On the unstable side, Dolich and Goodrick \cite{DoGo} have shown that the ordered group $(\Z,+,<,0)$ has no proper \emph{strong} expansions (which includes expansions of finite dp-rank). Concerning \emph{reducts} of Presburger arithmetic, a recent result of the first author \cite{CoPA} is that there are no structures strictly between $(\Z,+,0)$ and $(\Z,+,<,0)$. In a different direction, Kaplan and Shelah \cite{KaSh} show that if $P=\{z\in\Z:|z|\text{ is prime}\}$ then $(\Z,+,0,P)$ is unstable and, assuming a fairly strong conjecture in number theory, $(\Z,+,0,P)$ is supersimple of $SU$-rank $1$ (see also Remark \ref{rem:misc}$(3)$ below). 

The investigation of stable expansions of $(\Z,+,0)$ also fits naturally into the general question of when good properties of a structure are preserved after adding a new predicate. For example Pillay and Steinhorn \cite{PiStD} proved that there are no proper o-minimal expansions of $(\N,<)$, while Marker \cite{MaSucc} exhibited proper strongly minimal expansions of $(\N,x\mapsto x+1)$. Zilber \cite{ZiCU} showed that there are proper $\omega$-stable expansions of the complex field $(\C,+,\cdot,0,1)$ (in particular, adding a predicate for the roots of unity), while Marker \cite{MaComp} proved that there are no proper stable expansions of $(\C,+,\cdot,0,1)$ by a semialgebraic set.

Even more generally, Theorem \ref{thm:main} fits into the investigation of when stronger forms of stability can be proved for stable groups satisfying various assumptions on definable subgroups. For example, in \cite{BalPi}, Baldwin and Pillay prove that if $G$ is superstable of finite $U$-rank, and $G$ has no proper connected type-definable normal abelian subgroups, then $G$ is $\omega$-stable. In \cite{Gag}, Gagelman proves that if $G$ is superstable of finite $U$-rank and satisfies the descending chain condition on definable subgroups, then $G$ is $\omega$-stable. It would be interesting to know if the finiteness conditions on weight and $U$-rank in Theorem \ref{thm:main} can be relaxed to obtain a characterization of superstable groups of a similar flavor. In particular, it is well known that if $G$ is a superstable group, then every type in $G$ has finite weight (i.e. $G$ is \emph{strongly stable}) and $G$ has no infinite descending $<_\infty$-chains of definable subgroups (i.e. $G$ satisfies the \emph{superstable descending chain condition}). Therefore, we ask the following question, which is an analog of Theorem \ref{thm:main} for superstable groups.

\begin{question}
Suppose $G$ is a strongly stable group satisfying the superstable descending chain condition. Is $G$ is superstable?
\end{question}

We end with some important remarks.

\begin{remark}\label{rem:misc}$~$
\begin{enumerate}[$(1)$]
\item Many of the results above on $(\Z,+,0)$ do not hold if one considers expansions of structures elementarily equivalent to $(\Z,+,0)$. For example, there are models $(M,+,0)$ of $\Th(\Z,+,0)$ with proper stable expansions of finite $U$-rank.
\item Theorem \ref{thm:mainZ} also holds with inp-rank in place of dp-rank, since these ranks coincide in the stable case (see \cite{AdStrong}). Therefore the theorem can be applied in the more general class of NTP$_2$ theories.
\item Fact \ref{fact:PaSk} does not hold if stable is replaced by simple. For example, by work of Chatzidakis and Pillay \cite{ChzPi}, there are ``generic" subsets  $P\subseteq\Z$ such that $(\Z,+,0,P)$ is unstable, but supersimple of $SU$-rank $1$.
\end{enumerate}
\end{remark}

\section{Preliminaries}\label{sec:prelim}

The purpose of this section is to collect the preliminary tools and facts that we will need in the proofs of our main results. Our intent is to include sufficient detail so as to make this paper accessible to a wider audience beyond those researchers well-versed in stability theory. For example, Lemma \ref{lem:RW} and Proposition \ref{prop:indec} are folkloric facts, which seem to be used primarily in the superstable context, and to not appear in the literature in more general settings. Therefore we have included proofs  suitable for the general stable case.

Throughout this section, $T$ is a stable first-order theory, and we assume $T=T^{\eq}$. We work in a sufficiently saturated monster model $\M$ of $T$, and use letters $A,B,\ldots$ for small parameter sets in $\M$, where a parameter set $A$ is \emph{small} (written $A\subset\M$) if $\M$ is $|T(A)|^+$-saturated. In general, a cardinal $\kappa$ is \emph{small} or \emph{bounded} if $\M$ is $\kappa^+$-saturated. We use letters $X,Y,\ldots$ for definable or type-definable sets, and we always identify such a set $X$ with its set of realizations $X(\M)$ in the monster model.  As usual, by a \emph{type-definable} set we mean an intersection of a small collection of definable sets.  Given a type $p$, and a type-definable set $X$, we write $p\models X$ if $p$ extends a type defining $X$.  We use $\ind$ for the nonforking independence relation in $T$. We assume familiarity with stability and $U$-rank. 

\begin{definition}\label{def:weight}
$~$
\begin{enumerate}
\item  Given a sequence $(\bar{b}_i)_{i\in I}$ of tuples and $C\subset\M$, we say $(\bar{b}_i)_{i\in I}$ is \textbf{$C$-independent} if $\bar{b}_i\ind_C \{\bar{b}_j:j\neq i\}$ for all $i\in I$.
\item Given $C\subset\M$ and $p\in S(C)$, define the \textbf{weight of $p$}, denoted $\wt(p)$, to be the supremum over cardinals $\kappa$ for which there is some $B\supseteq C$, a realization $\bar{a}\models p$, and a $B$-independent sequence $(\bar{b}_i)_{i<\kappa}$ such that $\bar{a}\ind_C B$ and $\bar{a}\nind_B \bar{b}_i$ for all $i<\kappa$.
\item Let $\rk$ denote either $U$-rank or weight.
\begin{enumerate}[$(i)$]
\item If $\bar{a}\in\M$ and $C\subset\M$ then $\rk(\bar{a}/C)$ denotes $\rk(\tp(\bar{a}/C))$.
\item If $X$ is type-definable, then $\rk(X)=\sup\{\rk(p):p\models X\}$. 
\end{enumerate}
\end{enumerate}
\end{definition}

The final notion of rank discussed in the introduction is \emph{dp-rank}, which we calculate for type-definable sets in the same way. In particular, if $X$ is type-definable then $\dpr(X)=\sup\{\dpr(p):p\models X\}$ where we set $\dpr(p)$ to be the supremum over cardinals $\kappa$ such that the relation ``$\dpr(p)\geq\kappa$" holds, as defined in \cite[Chapter 4]{Sibook}. We are justified in avoiding the full definition of dp-rank because of the following fact about stable theories.

\begin{fact}[\textnormal{\cite{AdStrong}, \cite{OnUs}, \cite{Sibook}}]\label{fact:dpwt}
If $X\subseteq\M$ is type-definable then $\wt(X)=\dpr(X)$.
\end{fact}

 We will use the following basic properties of $U$-rank and weight.

\begin{fact}\label{fact:RW}
Let $\rk$ denote either $U$-rank or weight.
\begin{enumerate}[$(a)$]
\item Given $\bar{a}\in\M$ and $C\subset\M$, $\rk(\bar{a}/C)=0$ if and only if $\bar{a}\in\acl(C)$.

\item Fix $\bar{a},\bar{b}\in\M$ and $C\subset\M$. If $\bar{a}\in\acl(\bar{b},C)$ then $\rk(\bar{a}/C)\leq\rk(\bar{b}/C)$.

\item Suppose $X$ is type-definable and $f$ is a definable function  with domain containing $X$. Then $\rk(f(X))\leq \rk(X)$. 
\end{enumerate}
\end{fact}
\begin{proof}
These are straightforward exercises. Parts $(b)$ and $(c)$ follow easily from part $(a)$ together with Lascar's inequality for $U$-rank (see \cite[Theorem 19.4]{Pobook}), and a sufficiently similar inequality for weight (see \cite[Lemma V.3.11(2)]{Shbook}). 
\end{proof}

In a superstable theory, the weight of a type $p$ is bounded by the sum of the integer coefficients in the Cantor normal form of $U(p)$ (see \cite[Theorem 19.9]{Pobook}). In particular, one has $\wt(p)\leq U(p)$, which still holds for stable theories in general. 

\begin{lemma}\label{lem:RW}
If $C\subset\M$ and $p\in S(C)$, then $\wt(p)\leq U(p)$.
\end{lemma}
\begin{proof}
Fix $p\in S(C)$.  Suppose we have a set $B\supseteq C$, a realization $\bar{a}\models p$, and a $B$-independent sequence $(\bar{b}_i)_{i<\kappa}$, for some cardinal $\kappa$, such that $\bar{a}\ind_C B$ and $\bar{a}\nind_B \bar{b}_i$ for all $i<\kappa$. We prove $U(\bar{a}/B)\geq\kappa$, which implies $U(p)\geq\kappa$. Given $i\leq \kappa$, define $B_i=B\cup \{\bar{b}_j:i\leq j\}$ (so $B_\kappa=B$). We prove, by induction on $i\leq\kappa$, that $U(\bar{a}/B_i)\geq i$. Given this, we will then have $U(\bar{a}/B)=U(\bar{a}/B_{\kappa})\geq\kappa$. 

The base case is trivial; so suppose $\lambda\leq\kappa$ is a limit ordinal and $U(\bar{a}/B_i)\geq i$ for all $i<\lambda$. For any $i<\lambda$, we have $B_\lambda\subseteq B_i$, and so $U(\bar{a}/B_\lambda)\geq U(\bar{a}/B_i)\geq i$. Therefore $U(\bar{a}/B_\lambda)\geq\lambda$. Finally, fix $i<\kappa$ and suppose $U(\bar{a}/B_i)\geq i$. Since $B_{i+1}\ind_B\bar{b}_i$ and $\bar{a}\nind_B \bar{b}_i$, we have $\bar{a}\nind_{B_{i+1}}\bar{b}_i$ by transitivity. Therefore $U(\bar{a}/B_{i+1})\geq i+1$.
\end{proof}

For general stable theories, Lemma \ref{lem:RW} is the most one can say concerning the relationship between weight and $U$-rank for arbitrary types (see Example \ref{ex:main}). However, when working ``close" to types of $U$-rank $1$, weight and $U$-rank coincide. This will be a key tool in the proof of our main result.

\begin{proposition}\label{prop:RW1}
Fix $C\subset\M$, and suppose $X\subseteq\M$ is such that $U(a/C)\leq 1$ for all $a\in X$. If $\bar{b}$ is a finite tuple in $\acl(XC)$ then $U(\bar{b}/C)=\wt(\bar{b}/C)$.
\end{proposition}
\begin{proof}
First, we observe that by Lascar's inequality and Fact \ref{fact:RW}$(b)$, $U(\bar{b}/C)$ exists (and is in fact finite) for any finite tuple $\bar{b}$ from $\acl(XC)$. In particular, for any such $\bar{b}$ and any $C\subseteq B\subseteq A$, we have $\bar{b}\ind_B A$ if and only if $U(\bar{b}/A)=U(\bar{b}/B)$. 

 Fix $\bar{b}\in\acl(XC)$. By Fact \ref{fact:RW}$(a)$, we may assume that some coordinate of $\bar{b}$ is not in $\acl(C)$. Fix $\bar{a}=(a_1,\ldots,a_n)\in X$, algebraically independent over $C$, with $\bar{b}\in\acl(\bar{a},C)$. Let $k\leq n$ be maximal such that, for some $i_1<\ldots<i_k\leq n$, we have $\bar{b}\ind_C a_{i_1}\ldots a_{i_k}$ (since $\bar{b}\in \acl(\bar{a},C)\backslash\acl(C)$ we must have $k<n$, and it is possible that $k=0$). Without loss of generality, assume $\bar{b}\ind_C a_1 \ldots a_k$. Let $\bar{a}_1=(a_1,\ldots,a_k)$ and $\bar{a}_2=(a_{k+1},\ldots,a_n)$. Since $\bar{b}\ind_C \bar{a}_1$, we have $U(\bar{b}/C,\bar{a}_1)=U(\bar{b}/C)$ and $\wt(\bar{b}/C,\bar{a}_1)=\wt(\bar{b}/C)$ (see, e.g.,  \cite[Lemma V.3.11]{Shbook}). So to prove the result, it suffices to show $U(\bar{b}/C,\bar{a}_1)=\wt(\bar{b}/C,\bar{a}_1)$.

Since $\bar{a}_2$ is algebraically independent over $C\bar{a}_1$, we have $U(a_i/C,\bar{a}_1)=1=\wt(a_i/C,\bar{a}_1)$ for all $k< i\leq n$ by Fact \ref{fact:RW}$(a)$ and Lemma \ref{lem:RW}. It follows from Lascar's inequality and \cite[Lemma V.3.11(1)]{Shbook} that $U(\bar{a}_2/C,\bar{a}_1)=|\bar{a}_2|=\wt(\bar{a}_2/C,\bar{a}_1)$. So to prove $U(\bar{b}/C,\bar{a}_1)=\wt(\bar{b}/C,\bar{a}_1)$, it suffices to show $U(\bar{b}/C,\bar{a}_1)=U(\bar{a}_2/C,\bar{a}_1)$ and $\wt(\bar{b}/C,\bar{a}_1)=\wt(\bar{a}_2/C,\bar{a}_1)$. Since $\bar{b}\in\acl(\bar{a}_2,\bar{a}_1,C)$, it suffices by Fact \ref{fact:RW}$(b)$ to show $\bar{a}_2\in\acl(\bar{b},\bar{a}_1,C)$. 

For a contradiction, suppose there is $k<i\leq n$ such that $a_i\not\in\acl(\bar{b},\bar{a}_1,C)$. Then  $U(a_i/C,\bar{a}_1,\bar{b})=1=U(a_i/C,\bar{a}_1)$ and so $a_i\ind_{C\bar{a}_1}\bar{b}$. Since $\bar{b}\ind_C \bar{a}_1$, we have $\bar{b}\ind_C\bar{a}_1a_i$ by symmetry and transitivity. This contradicts the maximality of $k$.
\end{proof}

\begin{remark}
The notion of weight also behaves nicely in simple theories. For example, after replacing all occurrences of $U$-rank with $SU$-rank, the statements of Fact \ref{fact:RW}, Lemma \ref{lem:RW}, and Proposition \ref{prop:RW1} hold when $T$ is simple (with identical proofs).
\end{remark}

We now turn to stable groups. Recall the following classical results.

\begin{fact}\label{fact:SG}
Let $G$ be a group definable in a model of a stable theory.
\begin{enumerate}[$(a)$]
\item \textnormal{(Baldwin-Saxl, see \cite[Proposition 1.4]{PoStG})} Let $\{H_i:i\in I\}$ be a family of uniformly definable subgroups of $G$, and set $H=\bigcap_{i\in I}H_i$. Then $H=\bigcap_{i\in I_0}H_i$ for some finite $I_0\subseteq I$. In particular, $H$ is definable.
\item \textnormal{(Poizat, see \cite[Theorem 5.17]{PoStG})} Any type-definable subgroup of $G$ is the intersection of at most $|T|$ many definable subgroups of $G$.
\end{enumerate}
\end{fact}

\begin{remark}
Unlike the previous preliminaries, these facts do not immediately go through through if $T$ is only assumed to be simple. Indeed, there are simple unstable groups where part $(a)$ fails \cite[Example 1]{WaSimG}. On the the other hand, whether part $(b)$ holds for groups definable in simple theories is a well known open question.
\end{remark}

For the rest of this section, when we say $G$ is a \emph{stable group}, we mean $G=G(\M)$ is a group definable in the monster model $\M$ of a stable theory $T=T^{\eq}$. Given a stable group $G$, we let $G^0$ denote the \emph{connected component of $G$}, which is the intersection of all definable subgroups of $G$ of finite index. By stability (e.g., Fact \ref{fact:SG}), $G^0$ is the intersection of at most $|T|$ many definable subgroups of $G$ of finite index, and hence  is type-definable (over the same parameters used to define $G$). We say $G$ is \emph{connected} if $G=G^0$.

For the sake of clarity, it is worth making a few remarks concerning weight and $U$-rank in stable groups. In particular, given a definable group $G$ and $A\subset\M$, we let $S_G(A)$ denote the space of complete types, over parameters in $A$, which contain a formula defining $G$. Then, if $\rk$ denotes either $U$-rank or weight, we can express $\rk(G)$ as
$$
\rk(G)=\sup\{\rk(p):p\in S_G(A)\text{ for some $A\subset\M$}\}.
$$
We say $G$ \emph{has finite $U$-rank} (respectively, \emph{finite weight}) if $U(G)<\omega$ (respectively, $\wt(G)<\omega$). 

If $G$ is stable then $U(p)=U(G)$ for any generic type $p$ in $G$ (see \cite[Lemma III.4.5$(i)$]{BeLa}). On the other hand, it is possible that all types in $G$ have finite weight, but $\wt(G)$ is not finite (e.g. Example \ref{ex:main}$(1)$ below). Since our focus is on the case that $\wt(G)$ is finite, we will not concern ourselves with this situation.

\begin{remark}\label{rem:ex}
When considering examples of stable groups, it is often the case that the group $G$ is the whole structure (i.e. defined by the formula $x=x$). Therefore, given a group $G=(G,\cdot,1,\ldots)$, when we speak of the $U$-rank or weight of $G$, we continue to mean as calculated in a monster model according the definitions and conventions above.
\end{remark}

The following examples illustrate some of the possible variety concerning weight, $U$-rank, and $<_\infty$-chains in stable groups. 

\begin{example}\label{ex:main}$~$
\begin{enumerate}
\item Let $G\models \Th(\Z,+,0,\{2^n:n\in\N\})$. Then $G$ is superstable of $U$-rank $\omega$ (see \cite{PaSk}, \cite{PoZ}). Thus every type in $G$ has finite weight (i.e. $G$ is strongly stable), but $G$ does not have finite weight by Theorem \ref{thm:main}.

\item Fix an integer $n>0$ and let $G\models\Th(\Q^n,+,0,(H_k)_{k<n})$ where, for each $k<n$, $H_k=\Q^k\times \{0\}^{n-k}$. We have a sequence $(E_k)_{k<n}$ of definable equivalence relations, given by $E_k(x,y)\leftrightarrow x-y\in H_k$. Given $a,b\in G$, let $d(a,b)=\min\{k<n: E_k(a,b)\}$. Then $d$ is an ultrametric on $G$, taking values in $\{0,1,\ldots,n\}$; and nonforking independence is characterized by: $A\ind_C B$ if and only if, for all $a\in \acl(AC)$, $d(a,\acl(BC))=d(a,\acl(C))$ (where algebraic closure is the same as in $\Th(\Q^n,+,0)$). Using this, one may verify that $G$ is superstable of $U$-rank $n$ and weight $1$. 

\item Let $G\models\Th(\Q^{\omega},+,0,(H_n)_{n<\omega})$ where, for each $n<\omega$, $H_n=\Q^n\times \{0\}^\omega$. Using a similar argument as in part (2), one may show that $G$ is superstable of $U$-rank $\omega$ and weight $1$.
\item Let $G\models\Th(\Q^{\omega},+,0,(K_n)_{n<\omega})$ where, for each $n<\omega$, $K_n=\{0\}^n\times \Q^\omega$. Then $G$ is strictly stable of weight $1$ (this is again similar to part $(2)$).
\end{enumerate}
\end{example}

Our final preliminary tools concern indecomposable sets in stable groups.

\begin{definition}
Let $G$ be a stable group. A type-definable set $X\subseteq G$ is \textbf{indecomposable} if, for all type-definable subgroups $H\leq G$, either $X/H$ is unbounded or $|X/H|=1$ (where $X/H=\{xH:x\in X\}$).
\end{definition}

\begin{proposition}\label{prop:indec}
Let $G$ be a stable group. Fix $A\subset \M$ and a stationary type $p\in S_G(A)$. Let $X=p(\M)$. Then $X\subseteq G$ is indecomposable.
\end{proposition}
\begin{proof}
Let $\mathcal{F}$ denote the family of type-definable subgroups $H\leq G$ such that $X/H$ is bounded. Let $H_0$ be the intersection of the elements of $\mathcal{F}$. Using Fact \ref{fact:SG}, it is a standard exercise to show that $H_0$ is a type-definable subgroup of $G$ and $X/H_0$ is bounded (i.e. $H_0\in\mathcal{F}$). Note that $A$-invariance of $X$ implies $A$-invariance of $H_0$, and so $H_0$ is type-definable over $A$.

Let $\tilde{p}\in S_G(\M)$ be the unique global nonforking extension of $p$. Let $C\subset X$ be a bounded set such that $X/H_0=\{cH_0:c\in C\}$, and fix a realization $u\in G$ of $\tilde{p}|_{AC}$. Then $u\in X$, and so $u\in cH_0$ for some $c\in C$, which means $\tilde{p}\models cH_0$. If $f\in\Aut(\M/A)$ then, by $A$-invariance of $H_0$ and $\tilde{p}$, we have $\tilde{p}\models f(c)H_0$, and so $f(cH_0)=f(c)H_0=cH_0$. Consequently, $cH_0$ is type-definable over $A$, and so $p\models cH_0$. Therefore $X\subseteq cH_0$, which implies $X\subseteq cH$ for all $H\in\mathcal{F}$.
\end{proof}

A well-known result of Berline and Lascar is the Indecomposability Theorem for superstable groups \cite[Theorem V.3.1]{BeLa}.  In order to use this result without the assumption of superstability, we state the following corollary of its proof.

\begin{fact}\label{thm:BeLa}
Suppose $G$ is a stable group and $\{X_i:i\in I\}$ is a family of indecomposable type-definable subsets of $G$, each containing $1$. Given  $n>0$ and $\sigma=(i_0,\ldots,i_n)\in I^{<\omega}$, let $X_\sigma=X_{i_0}\cdot X_{i_1}\cdot\ldots\cdot X_{i_n}$. Assume that there is a uniform finite bound on $U(X_\sigma)$, where $\sigma$ ranges over $I^{<\omega}$. Then $\bigcup_{i\in I}X_i$ generates a connected type-definable subgroup $H$ of $G$. In particular, there are $i_0,\ldots,i_n$ such that $H=X_{i_0}\cdot\ldots\cdot X_{i_n}\cdot X_{i_n}^{\text{-}1}\cdot\ldots\cdot X_{i_0}^{\text{-}1}$. 
\end{fact}

\section{Proofs of the main results}\label{sec:proofs}

As in the previous section, when we say $G$ is a \emph{stable group} we mean $G=G(\M)$ is a group definable in the monster model $\M$ of a stable theory $T=T^{\eq}$. Toward the proofs of Theorems \ref{thm:mainZ}, \ref{thm:main}, and \ref{thm:all}, we start with the following technical lemma concerning definable subgroups of infinite stable groups of finite weight.

\begin{lemma}\label{lem:main}
Let $G$ be an infinite stable group of finite weight.
\begin{enumerate}[$(a)$]
\item  There is an infinite connected type-definable normal subgroup $H\leq G$, with $U(H)=\wt(H)$. 
\item Assume $G$ has no infinite descending $<_\infty$-chains of definable normal subgroups. If $K<_\infty G$ is definable and normal in $G$, then there is a definable normal subgroup $L\leq G$ such that $K<_\infty L$ and $U(G/K)\leq \wt(G/K)\oplus U(G/L)$. 
\end{enumerate}
\end{lemma}
\begin{proof}
Part $(a)$. Fix a stationary type $p\in S_G(A)$, for some $A\subset\M$, such that $U(p)=1$. For example, choose $p$ minimal in the fundamental order among non-algebraic types in $S_G(A)$ (with $A$ varying over small parameter sets in $\M$), and then replace $p$ by a nonforking extension to a model. 

Let $Y=p(\M)$. Then $Y\subseteq G$ is indecomposable by Proposition \ref{prop:indec}. Fix some $u\in Y$, and set $X=u^{\text{-}1} Y$. Given $g\in G$, let $X_0^g=gXg^{\text{-}1}$ and $X_1^g=g^{\text{-}1}X^{\text{-}1}g=(X_0^g)^{\text{-}1}$. Then $\{X_i^g:g\in G,~i\in\{0,1\}\}$ is a family of indecomposable type-definable subsets of $G$, each of which contains $1$. By Fact \ref{fact:RW}$(c)$, $U(X_i^g)=1$ for all $g\in G$ and $i\in\{0,1\}$. Fix a sequence $\sigma=(g_0,\ldots,g_n)$ of elements of $G$, and set $X_\sigma=X_0^{g_0}\cdot\ldots\cdot X_0^{g_n}\cdot X_1^{g_n}\cdot\ldots\cdot X_1^{g_0}$. In particular, $X_\sigma\subseteq \acl(\bigcup_{t=0}^n X^{g_t}_0\cup X^{g_t}_1)$ and so, by Proposition \ref{prop:RW1}, $U(q)=\wt(q)$ for any $q\models X_\sigma$. Therefore $U(X_\sigma)=\wt(X_\sigma)$, and so $U(X_\sigma)\leq\wt(G)$ by Fact \ref{fact:RW}$(c)$.

Now we may apply Fact \ref{thm:BeLa} to conclude that $\bigcup\{X_i^g:g\in G,~i\in\{0,1\}\}$ generates an infinite connected type-definable subgroup $H$ of $G$, which is normal by construction. Moreover, $H=X_\sigma$ for some $\sigma\in G^{<\omega}$, and so $U(H)=\wt(H)$.

Part $(b)$. Let $K<_\infty G$ be definable and normal. We use $\rho$ to denote the pullback function on subgroups of $G/K$, i.e., given $H\leq G/K$ define $\rho(H)=\{g\in G:gK\in H\}\leq G$.

By assumption and Fact \ref{fact:RW}$(c)$, $G/K$ is an infinite stable group of finite weight.  By part $(a)$ applied to $G/K$, there is an infinite connected type-definable normal subgroup $H\leq G/K$, with $U(H)=\wt(H)$. Since $H$ is type-definable, it is the intersection of a bounded family $(H_i)_{i\in I}$ of definable subgroups by Fact \ref{fact:SG}$(b)$. Since $H$ is normal we may use Fact \ref{fact:SG}$(a)$ to replace each $H_i$ with $\bigcap_{g\in G/K}gH_ig^{\text{-}1}$, and thus assume $H$ is the intersection of a bounded family of definable normal subgroups of $G/K$. Now, $G/K$ has no infinite descending $<_\infty$-chains of definable normal subgroups since such a chain would pull back via $\rho$ to a chain in $G$. It follows that there is a definable normal subgroup $J$ of $G/K$ such that $H\leq J$ and $[J:H]$ is bounded. Since $H$ is type-definable and connected we then have $H=J^0$, which implies $U(J)=U(H)=\wt(H)$ (see, e.g., \cite[Sections III.4, IV.3]{BeLa}). By Fact \ref{fact:RW}$(c)$, $U(J)\leq\wt(G/K)$.

Now let $L=\rho(J)$. Then $L$ is a definable normal subgroup of $G$ and, since $J$ is infinite, $K<_\infty L$. By definition of $L$, the groups $G/L$ and $(G/K)/J$ are definably isomorphic and so, by Lascar's inequality for cosets \cite[Corollary III.8.2]{BeLa},
\[
U(G/K)\leq U(J)\oplus U(G/L)\leq \wt(G/K)\oplus U(G/L).\qedhere
\]
\end{proof}

We now prove the main results stated in the introduction.

\begin{proof}[Proof of Theorem \ref{thm:all}]
Let $G$ be a stable group of finite weight, with no infinite $<_\infty$-chains of definable normal subgroups. We will use Lemma \ref{lem:main}$(b)$ to construct an ascending $<_\infty$-chain of definable normal subgroups of $G$. By assumption, this construction must terminate at some finite step, at which point we will make the desired conclusions (claims $(i)$ and $(ii)$ in the statement of the theorem).

To start the construction, let $K_0=\{1\}$. Now fix $m<\omega$ and suppose we have constructed definable normal subgroups $K_0<_\infty\ldots<_\infty K_m\leq G$ such that $U(G)\leq U(G/K_m)\oplus\sum_{i<m}\wt(G/K_i)$. If $G/K_m$ is finite then we terminate the construction. Otherwise, if $K_m<_\infty G$ then we use Lemma \ref{lem:main}$(b)$ to find a definable normal subgroup $K_{m+1}\leq G$ such that $K_m<_\infty K_{m+1}$ and $U(G/K_m)\leq U(G/K_{m+1})\oplus \wt(G/K_m)$. By induction, $U(G)\leq U(G/K_{m+1})\oplus\sum_{i\leq m}\wt(G/K_i)$. 

Since $G$ has no infinite ascending $<_\infty$-chains of normal subgroups, the above construction must terminate at some $\hat{m}<\omega$, meaning that $G/K_{\hat{m}}$ is finite. By construction and Fact \ref{fact:RW}, $U(G)\leq \sum_{i<\hat{m}}\wt(G/K_i)\leq \hat{m}\wt(G)$. Thus $U(G)$ is finite which, by Lascar's inequality for cosets, immediately yields claim $(i)$. For claim $(ii)$, let $n$ be the maximal length of a $<_\infty$-chain of definable normal subgroups of $G$ (note that $n$ exists by $(i)$). We must have $\hat{m}\leq n$ and so $U(G)\leq n\wt(G)$.  
\end{proof}

\begin{remark}
Note that in the proof of Theorem \ref{thm:all}, the assumption of no infinite descending $<_\infty$-chains of definable normal subgroups is used when applying Lemma \ref{lem:main}$(b)$. The stable groups described in parts $(3)$ and $(4)$ of Example \ref{ex:main} illustrate that all assumptions in Theorem \ref{thm:all} are necessary.
\end{remark}

As outlined in the introduction, Theorem \ref{thm:main} follows immediately from Theorem \ref{thm:all} and standard facts.

\begin{proof}[Proof of Theorem \ref{thm:main}]
$(ii)\Rightarrow (iii)$ is trivial, and $(iii)\Rightarrow (i)$ is by Theorem \ref{thm:all}. For $(i)\Rightarrow(ii)$, first recall that finite $U$-rank implies finite weight by Lemma \ref{lem:RW}. Moreover, if $G$ is superstable of finite $U$-rank then it follows from Lascar's inequality for cosets that $G$ has no infinite $<_\infty$-chains of definable subgroups. 
\end{proof}

Finally, we apply Theorem \ref{thm:main} to prove our main result concerning $(\Z,+,0)$, namely that there are no proper stable expansions of $(\Z,+,0)$ of finite dp-rank.

\begin{proof}[Proof of Theorem \ref{thm:mainZ}]
Suppose $P\subseteq\Z^n$ is definable in a finite dp-rank expansion of $(\Z,+,0)$ and $(\Z,+,0,P)$ is stable. We want to show $P$ is definable in $(\Z,+,0)$. We work in $T=\Th(\Z,+,0,P)$, and let $G$ be a sufficiently saturated model of $T$.  Since dp-rank cannot increase after taking a reduct, $G$ has finite dp-rank, and thus  finite weight by Fact \ref{fact:dpwt}. We claim that $G$ has no nontrivial definable subgroup of infinite index. Indeed, otherwise in $\Z$ we obtain a family $(H_n)_{n<\omega}$ of uniformly definable nontrivial subgroups of $\Z$ such that $H_n$ has index at least $n$. In particular, the intersection $H=\bigcap_{n\in\omega}H_n$ has infinite index in $\Z$, and thus $H=\{0\}$. But by Fact \ref{fact:SG}$(a)$, $H$ is equal to a finite subintersection, which is a contradiction since the intersection of finitely many nontrivial subgroups of $\Z$ is infinite. Now we may apply Theorem \ref{thm:main} to $G$ and conclude $U(G)$ is finite. Then $T$ is superstable of finite $U$-rank and so $P$ is definable in $(\Z,+,0)$ by Fact \ref{fact:PaSk}.
\end{proof}

\subsection*{Acknowledgements} 
We are grateful to Rizos Sklinos and Erik Walsberg for their comments on an earlier draft. We also thank the referee for several helpful comments and suggestions, which greatly improved the final version.

\bibliographystyle{amsplain}

\begin{thebibliography}{10}

\bibitem{AdStrong}
Hans Adler, \emph{Strong theories, burden, and weight}, preprint,
  \url{http://www.logic.univie.ac.at/~adler/docs/strong.pdf}, 2007.

\bibitem{BalPi}
J.~T. Baldwin and A.~Pillay, \emph{Semisimple stable and superstable groups},
  Ann. Pure Appl. Logic \textbf{45} (1989), no.~2, 105--127, Stability in model
  theory, II (Trento, 1987). \MR{1044119 (91f:03064)}

\bibitem{BeLa}
Ch. Berline and D.~Lascar, \emph{Superstable groups}, Ann. Pure Appl. Logic
  \textbf{30} (1986), no.~1, 1--43, Stability in model theory (Trento, 1984).
  \MR{831435 (87k:03028)}
  
 \bibitem{BySk}
Ayala Byron and Rizos Sklinos, \emph{Fields definable in the free group},
  arXiv:1512.07922, 2015.
  
  \bibitem{ChzPi}
Z.~Chatzidakis and A.~Pillay, \emph{Generic structures and simple theories},
  Ann. Pure Appl. Logic \textbf{95} (1998), no.~1-3, 71--92. \MR{1650667
  (2000c:03028)}

\bibitem{CoPA}
Gabriel Conant, \emph{There are no intermediate structures between the group of
  integers and {P}resburger arithmetic}, arXiv:1603.00454, 2016.


\bibitem{DoGo}
Alfred Dolich and John Goodrick, \emph{Strong theories of ordered {A}belian
  groups}, Fund. Math. \textbf{236} (2017), no.~3, 269--296. \MR{3600762}

\bibitem{DGL}
Alfred Dolich, John Goodrick, and David Lippel, \emph{Dp-minimality: basic
  facts and examples}, Notre Dame J. Form. Log. \textbf{52} (2011), no.~3,
  267--288. \MR{2822489 (2012h:03102)}

\bibitem{FlGu}
Joseph Flenner and Vincent Guingona, \emph{Convexly orderable groups and valued
  fields}, J. Symb. Log. \textbf{79} (2014), no.~1, 154--170. \MR{3226016}

\bibitem{Gag}
Jerry Gagelman, \emph{A note on superstable groups}, J. Symbolic Logic
  \textbf{70} (2005), no.~2, 661--663. \MR{2140052 (2006a:03046)}
  
  \bibitem{KaSh}
Itay Kaplan and Saharon Shelah, \emph{Decidability and classification of the
  theory of integers with primes}, J. Symb. Logic, accepted, arXiv:1601.07099, 2016.
  
  \bibitem{MaSucc}
David Marker, \emph{A strongly minimal expansion of {$(\omega,s)$}}, J.
  Symbolic Logic \textbf{52} (1987), no.~1, 205--207. \MR{877867 (88h:03048)}

\bibitem{MaComp}
\bysame, \emph{Semialgebraic expansions of {${\bf C}$}}, Trans. Amer. Math.
  Soc. \textbf{320} (1990), no.~2, 581--592. \MR{964900 (90k:03034)}
  


\bibitem{OnUs}
Alf Onshuus and Alexander Usvyatsov, \emph{On dp-minimality, strong dependence
  and weight}, J. Symbolic Logic \textbf{76} (2011), no.~3, 737--758.
  \MR{2849244 (2012m:03082)}

\bibitem{PaSk}
Daniel Palac\'{i}n and Rizos Sklinos, \emph{Superstable expansions of free
 abelian groups}, Notre Dame J. Form. Log., to appear, available: arXiv
  1405.0568.
  

\bibitem{PiFFG}
Anand Pillay, \emph{Forking in the free group}, J. Inst. Math. Jussieu
  \textbf{7} (2008), no.~2, 375--389. \MR{2400726 (2009f:20033)}
  
  
\bibitem{PiStD}
Anand Pillay and Charles Steinhorn, \emph{Discrete {$o$}-minimal structures},
  Ann. Pure Appl. Logic \textbf{34} (1987), no.~3, 275--289, Stability in model
  theory (Trento, 1984). \MR{899083 (88j:03023)}


\bibitem{Pobook}
Bruno Poizat, \emph{A course in model theory}, Universitext, Springer-Verlag,
  New York, 2000, An introduction to contemporary mathematical logic,
  Translated from the French by Moses Klein and revised by the author.
  \MR{1757487 (2001a:03072)}

\bibitem{PoStG}
\bysame, \emph{Stable groups}, Mathematical Surveys and Monographs, vol.~87,
  American Mathematical Society, Providence, RI, 2001, Translated from the 1987
  French original by Moses Gabriel Klein. \MR{1827833 (2002a:03067)}

\bibitem{PoZ}
\bysame, \emph{Superg\'en\'erix}, J. Algebra \textbf{404} (2014), 240--270,
  {\`A} la m{\'e}moire d'{\'E}ric Jaligot. [In memoriam {\'E}ric Jaligot].
  \MR{3177894}
  
  \bibitem{Sela}
Z.~Sela, \emph{Diophantine geometry over groups {VIII}: {S}tability}, Ann. of
  Math. (2) \textbf{177} (2013), no.~3, 787--868. \MR{3034289}

\bibitem{Shbook}
Saharon Shelah, \emph{Classification theory and the number of nonisomorphic
  models}, second ed., Studies in Logic and the Foundations of Mathematics,
  vol.~92, North-Holland Publishing Co., Amsterdam, 1990. \MR{1083551
  (91k:03085)}
  
  \bibitem{Sibook}
Pierre Simon, \emph{A guide to {NIP} theories}, vol.~44, Cambridge University
  Press, 2015.

\bibitem{SiDP}
\bysame, \emph{On dp-minimal ordered structures}, J. Symbolic Logic
  \textbf{76} (2011), no.~2, 448--460. \MR{2830411 (2012e:03071)}
  
  \bibitem{WaSimG}
Frank Wagner, \emph{Groups in simple theories}, Logic {C}olloquium '01, Lect.
  Notes Log., vol.~20, Assoc. Symbol. Logic, Urbana, IL, 2005, pp.~440--467.
  \MR{2143908}
  
  \bibitem{ZiCU}
Boris Zilber, \emph{A note on the model theory of the complex field with roots
  of unity}, unpublished note,
  \url{https://people.maths.ox.ac.uk/zilber/publ.html}, 1990.

\end{thebibliography}
\end{document}